\documentclass[leqno,11pt,a4paper]{amsart}
\usepackage{bw}

\newcommand{\kqcyc}{kQ_\text{\textnormal{cyc}}}

\title{Type \texttt{0} walls for G-Hilb}

\title{Crepant resolutions, mutations, and the space of potentials}
\author[M.~Barker]{Mary Barker}
\address{Fred Hutchinson Cancer Center \\ Seattle \\ WA \\ 98109 \\ USA}
\email{marybarker103@gmail.com}

\author[B.~Standaert]{Benjamin Standaert}
\address{Department of Computer Science \\ Washington University in St.~Louis \\ St.~Louis \\ MO \\ 63130 \\ USA}
\email{b.g.standaert@wustl.edu}

\vs

\begin{document}

\maketitle

\begin{abstract}
    The McKay correspondence has had much success in studying resolutions of $3$-fold quotient singularities through a wide range of tools coming from geometry, combinatorics, and representation theory. We develop a computational perspective in this setting primarily realised through a web application to explore mutations of quivers with potential and crepant triangulations. We use this to study flops between different crepant resolutions of Gorenstein toric quotient singularities and find many situations in which the mutations of a quiver with potential classifies them. The application also implements key constructions of the McKay correspondence, including the Craw--Reid procedure and the process of associating a quiver to a toric resolution.
\end{abstract}

\section{Introduction}

\subsection{McKay correspondence} The McKay correspondence seeks to relate the representation theory of a finite subgroup $G\subseteq\on{SL}_n(\C)$ and the geometry of minimal resolutions of the corresponding quotient singularity $\C^n/G=\on{Spec}{\C[x_1,\dots,x_n]^G}$. Resolutions of $\C^n/G$ are very well-understood in dimension two and have been the subject of much study in dimension three. Crepant, a fortiori minimal, resolutions are known to exist in this setting by the landmark work of Bridgeland--King--Reid \cite{bkr_mck_01} and have been seen to all be realisable as moduli spaces $\M_\theta$ of $\theta$-stable quiver representations (see \cite{ci_flo_04} for the abelian case and the recent preprints \cite{yam_mod_22a,yam_mod_22b} for the general case). The quiver in question is the \emph{McKay quiver} $Q_G$ arising from the representation theory of $G$.

\vspace{1em}

The stability parameter $\theta$ inhabits a high dimensional vector space that decomposes into chambers where the isomorphism type of $\M_\theta$ is constant. There is much interest in understanding this chamber decomposition and its consequences for the birational geometry of the singularity $\C^n/G$ and its resolutions \cite{ci_flo_04,wor_wal_20,wor_wal_22,ns_flo_17,yam_mod_22b}.

\vspace{1em}

When $G\subseteq\on{SL}_3(\C)$ is a polyhedral group -- conjugate to a subgroup of $\on{SO}_3(\R)$ -- Nolla de Celis--Sekiya \cite{ns_flo_17} calculate the crepant resolutions of $\C^3/G$ and naturally relate them to mutations of the McKay quiver equipped with a ``potential'' coming from representation theory \cite{bsw_sup_10}. Quivers with potentials and their mutations were developed in a series of papers by Derksen--Weyman--Zelevinsky \cite{dwz_qui_08,dwz_qui_10}. In this setting a potential is a formal linear combination of cycles in $Q$, and there is again an involutive mutation operation producing from a choice of a quiver with potential and a node a new quiver with potential. One striking contrast to the usual Fomin--Zelevinsky \cite{fz_clu_02} theory of quiver mutations is that many more quivers with potential have finite mutation type than just the finite-type quivers. In this situation mutations for quivers correspond to flops in exceptional curves in geometry.

\vspace{1em}

Our starting point is to generalise the work of Nolla de Celis--Sekiya to other classes of subgroups of $\on{SL}_3(\C)$; in this paper to consider all finite abelian subgroups. In this case the singularity $\C^3/G$ and its crepant resolutions are toric and so the latter can be represented as triangulations of a triangle living in a certain lattice (see \S\ref{sec:abelian}).

\vspace{1em}

Outside of the polyhedral setting the McKay quiver is not the right gadget to use; for instance, it is almost exclusively in the polyhedral case that the vertices of $Q_G$ biject with exceptional curves in a crepant resolution of $\C^3/G$. The object we use instead is the \emph{curve quiver}. This nomenclature is not standard nor is our definition (see \S\ref{sec:cq}) but this quiver makes regular appearances in the homological minimal model program of Wemyss and collaborators \cite{wem_flo_18,wem_asp_14,dw_non_16}. We pose the following conjecture first considered in \cite[Rem.~4.7]{wor_wal_22}.

\begin{conjecture*} \label{conj:main}
    Let $G$ be a finite (abelian) subgroup of $\on{SL}_3(\C)$, and let $Q$ be the curve quiver for a crepant resolution of $\C^3/G$. Then there is a unique minimal potential $W$ on $Q$ such that the quivers with potential mutation-equivalent to $(Q,W)$ naturally biject with crepant resolutions of $\C^3/G$ in such a way that mutation in a node corresponds to flopping the associated curve.
\end{conjecture*}

We measure minimality of potentials by considering the sets of monomials that they are supported on under inclusion. Our computational methods allow us to probe this conjecture, which quickly becomes unfeasible for hand computation.

\begin{thm*}
Conjecture \ref{conj:main} is true for 74 cyclic subgroups of $\on{SL}_3(\C)$ listed in \S\ref{sec:list}.
\end{thm*}

\begin{remark*}
    It is worth comparing our conjectures and results to existing connections between mutations in the context of cluster algebras and triangulations on (oriented) surfaces \cite{lab_tri_16,wil_qui_21,gm_bra_17}.
\end{remark*}

We next outline our main computational tool -- the web application \texttt{qwp\_mutation} -- and our new results in the theory of quivers with potential.

\subsection{Computational methods for quivers with potential} 

Fix a quiver $Q$. For each potential $W$ on $Q$ we obtain a set of quivers with potential obtained by sequences of mutations from the pair $(Q,W)$. Invariants of this set allow us to stratify the space of potentials on $Q$. We investigate this construction for three invariants:

\begin{enumerate}
    \item the \emph{exchange number}: the number of distinct quivers obtained by iterated mutation from $(Q,W)$,
    \item the \emph{quiver set}: the set of quivers obtained by iterated mutation from $(Q,W)$,
    \item the \emph{exchange graph}: the quiver set viewed as a labelled graph by adding edges between nodes for each (one-step) mutation between quivers with potential.
\end{enumerate}

The exchange graph is a strong invariant in the following sense.

\begin{thm*}[Prop.~\ref{prop:minimal}]
    Let $Q$ be a quiver. Fix a graph $\Gamma$ labelled by quivers as in (iii). There is at most one minimal potential $W$ on $Q$ such that the exchange graph of $(Q,W)$ is identified with $\Gamma$.
\end{thm*}

This validates our computational approach: if we can locate a minimal potential that captures the resolutions of $\C^3/G$ in the sense of Conjecture \ref{conj:main} then we have found the unique such potential. We found heuristically that the quiver set is often sufficient to pin down this unique minimal potential, which is much more computationally efficient to use. We conjecture that this is always sufficient to use for this purpose.

\begin{question*}[Question \ref{conj:qs_eg}]
Given the curve quiver of a crepant triangulation $\mathcal{T}$, does every potential with quiver set equal to the set of curve quivers associated to triangulations obtained from $\mathcal{T}$ by flips further have the same exchange graph in the sense of Conjecture \ref{conj:main}?
\end{question*}

\vspace{1em}

The \texttt{qwp\_mutation} web application\footnote{Freely available for use at \url{https://marybarker.github.io/quiver_mutations}.} has several functions:

\begin{itemize}
    \item \textbf{Mutation:} Given the input of a quiver with potential, users can mutate in nodes of the quiver (or identify if this is not possible).
    \item \textbf{Quiver set:} Given the input of a quiver with potential $(Q,W)$, the application computes its quiver set.
    \item \textbf{Craw--Reid procedure:} Given the input of a triple $(a,b,c)\in\N^3$, the application produces the triangulation for $G\hilb\C^3$ (see Rem.~\ref{rem:moc} for the current mild implementation issues).
    \item \textbf{Curve quiver:} Given the input of a triple $(a,b,c)\in\N^3$, the application produces the curve quiver associated to the triangulation for $G\hilb\C^3$.
    \item \textbf{Verification:} Given the input of a triple $(a,b,c)\in\N^3$ and a potential $W$ on the associated curve quiver, the application can verify whether the exchange graph of $(Q,W)$ can be identified with the graph of crepant resolutions of $\C^3/G$, where $G=\frac{1}{a+b+c}(a,b,c)$, linked by flops as in Conjecture \ref{conj:main}.
\end{itemize}

See \S\ref{sec:app} for further discussion on efficiency, success rates, and examples of the utility of this web application. We hope that it will serve many researchers in finding computational evidence or inspiration for questions around the McKay correspondence and involving mutations of quivers with potential.

\subsection*{Acknowledgements} The authors are grateful to Tom Ducat for many helpful conversations in the early stages of this project. They are also grateful to Michael Wemyss and Franco Rota for their input. They drew much inspiration from Bernhard Keller's quiver mutation applet\footnote{Freely available at \url{https://webusers.imj-prg.fr/~bernhard.keller/quivermutation/}.} for which they thank him.

\section{Background}

\subsection{Quivers and mutations} \label{sec:qui_mut}

A quiver is a pair $Q=(Q_0,Q_1)$ of sets with two functions $t,h\colon Q_1\to Q_0$ that we refer to as the \emph{tail} and \emph{head} of an element of $Q_1$. We refer to $Q_0$ as the vertices and $Q_1$ as the arrows of $Q$. It is typical to identify $Q$ with the directed graph supported on node set $Q_0$ where each arrow $\alpha\in Q_1$ goes from $t(\alpha)$ to $h(\alpha)$. We refer to an arrow $\alpha$ with $t(\alpha)=h(\alpha)$ as a loop.

Given a quiver $Q$ with no $2$-cycles and a loop-less node $k\in Q_0$ Fomin--Zelevinsky \cite{fz_clu_02} define the \emph{mutation} $\mu_kQ$ of $Q$ at $k$ in the following way:

\begin{itemize}
\item for each pair of arrows $\alpha,\beta$ with $h(\alpha)=t(\beta)=k$ introduce a new arrow $\alpha\beta$ with $t(\alpha\beta)=t(\alpha)$ and $h(\alpha\beta)=h(\beta)$.
\item for each arrow $\gamma$ with either $t(\gamma)=k$ or $h(\gamma)=k$ remove $\gamma$ and introduce a new arrow $\gamma^*$ with $t(\gamma^*)=h(\gamma)$ and $h(\gamma^*)=t(\gamma)$.
\item for each new arrow $\alpha\beta$ or $\gamma^*$ that causes a $2$-cycle to be formed in $Q$ remove $\alpha\beta$ or $\gamma^*$ and a second arrow constituting a $2$-cycle with $\alpha\beta$ or $\gamma^*$.
\end{itemize}

\begin{example} Consider the quiver $Q$ depicted in Fig.~\ref{fig:mut_ex}(a). The mutation of $Q$ in the node labelled $\star$ is shown in Fig.~\ref{fig:mut_ex}(b).
\begin{figure}[h]
    \centering
\begin{tikzpicture}[scale=0.8]
\node (a) at (0,2){$\bullet$};
\node (o) at (0,0){$\star$};
\node (b) at (2,0){$\bullet$};
\node (c) at (0,-2){$\bullet$};
\node (d) at (-2,0){$\bullet$};

\draw[->] (a) to (o);
\draw[->] (o) to (b);
\draw[->] (o) to (d);
\draw[->] (c) to (o);
\draw[->] (d) to (a);
\draw[->] (b) to (c);

\node (a) at (6,2){$\bullet$};
\node (o) at (6,0){$\star$};
\node (b) at (8,0){$\bullet$};
\node (c) at (6,-2){$\bullet$};
\node (d) at (4,0){$\bullet$};

\draw[<-] (a) to (o);
\draw[<-] (o) to (b);
\draw[<-] (o) to (d);
\draw[<-] (c) to (o);
\draw[->] (c) to (d);
\draw[->] (a) to (b);

\node at (0,-3){(a)};
\node at (6,-3){(b)};
\end{tikzpicture}
    \caption{Example of quiver mutation}
    \label{fig:mut_ex}
\end{figure}
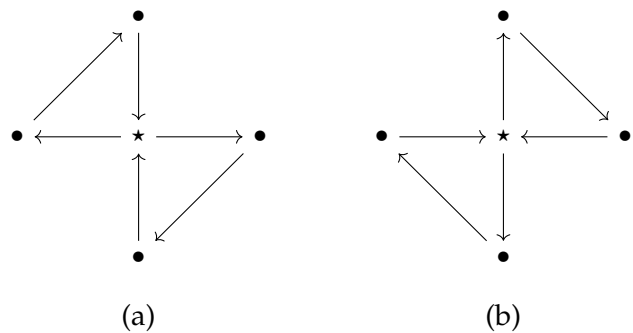
\end{example}

Mutation of quivers in this sense is intimately related to the theory of cluster algebras and associated representation theory and combinatorics \cite{fz_clu_02,kel_clu_08}.

\subsection{Quivers with potentials and mutations} \label{sec:qwp_mut}

We follow various authors \cite{ns_flo_17,dwz_qui_08} in extending beyond mutations of quivers without $2$-cycles to mutations of more general quivers augmented by a potential. To define quivers with potential we first construct the \emph{path algebra} of a quiver.

\begin{definition}
Let $Q=(Q_0,Q_1)$ be a quiver. Consider the free algebra $k\langle Q_0\cup Q_1\rangle$. We denote the algebra element corresponding to a node $i\in Q_0$ by $e_i$. The \emph{path algebra} $kQ$ of $Q$ with coefficients in a field $k$ is the noncommutative $k$-algebra defined as the quotient of $k\langle Q_0\cup Q_1\rangle$ by the two-sided ideal generated by
$$e_i\alpha\qquad\alpha e_j\qquad\alpha\beta \qquad e_p\alpha-\alpha\qquad \alpha e_q-\alpha$$
whenever $t(\alpha)=p\not=i,h(\alpha)=q\not=j$ and $h(\alpha)\not=t(\beta)$. In other words, the multiplication in $kQ$ captures concatenation of arrows in $Q$ with the node generators $e_i$ acting as idempotents picking out arrows beginning or ending at $i$. The field $k$ will play a largely auxiliary role in the following.
\end{definition}

Consider the subalgebra $\kqcyc\subseteq kQ$ generated by all cycles in $Q$:
$$\alpha_1\alpha_2\dots\alpha_r$$
such that $h(\alpha_i)=t(\alpha_{i+1})$ and $t(\alpha_1)=h(\alpha_r)$. We call elements of $\kqcyc$ \emph{potentials} on $Q$ with coefficients in $k$ and a pair $(Q,W)$ consisting of a quiver $Q$ and a potential $W$ on $Q$ a \emph{quiver with potential} (QwP). We say that a QwP is \emph{reduced} if it has no quadratic terms (i.e.~monomials of degree $2$).

The \emph{derivative} of a monomial $\gamma=\beta_1\dots\beta_r\in kQ$ in direction $\alpha\in Q_1$ is
$$\partial_\alpha\gamma=\begin{cases}
0 & \alpha\not=\beta_i \\
\beta_1\dots\beta_{i-1}\beta_{i+1}\dots\beta_r & \alpha=\beta_i
\end{cases}$$
Given a QwP we can define its \emph{Jacobian algebra}
$$P(Q,W):=kQ/\partial W$$
where $\partial W$ is the two-sided ideal generated by derivatives $\partial_\alpha W$ as $\alpha\in Q_1$. This is an important algebraic invariant of $(Q,W)$ containing many well-known classes of algebra like preprojective algebras \cite{rin_pre_98}. We call two QwPs \emph{Jacobian equivalent} if their Jacobian algebras are isomorphic.

We follow Derksen--Weyman--Zelevinsky and Nolla de Celis--Sekiya \cite{ns_flo_17,dwz_qui_08} in defining mutation for (many) QwPs as follows. Let $(Q,W)$ be a QwP where $Q$ is permitted $2$-cycles and let $k\in Q_0$ be a loopless node.
\begin{enumerate}
    \item Define $\wt{Q}$ as the quiver with node set $Q_0$ and edge set
    $$Q_1\setminus\{\alpha:h(\alpha)=k\}\cup\{\beta:t(\beta)=k\}$$
    along with new arrows
$$[\alpha\beta]\qquad\alpha^*\qquad\beta^*$$
whenever $h(\alpha)=k=t(\beta)$ where
$$h([\alpha\beta])=h(\beta)\quad t([\alpha\beta])=t(\alpha)\qquad h(\alpha^*)=t(\alpha)\quad t(\alpha^*)=h(\alpha)\qquad h(\beta^*)=t(\beta)\quad t(\beta^*)=h(\beta)$$
    \item Construct a potential $\wt{W}$ on $\wt{Q}$ from $W$ by converting any occurrence of $\alpha\beta$ to $[\alpha\beta]$ and appending all terms of the form
$$\beta^*\alpha^*[\alpha\beta]$$
    \item Let $(Q',W')$ be a reduced QwP that is Jacobian equivalent to $(Q,W)$. Define the mutation of $(Q,W)$ at node $k$ to be
    $$\mu_k(Q,W)=(Q',W').$$
\end{enumerate}

\cite[Cor.~5.4]{dwz_qui_08} shows that this operation is (suitably) well-defined, \cite[Thm.~5.7]{dwz_qui_08} shows that it is involutive, and \cite[Lem.~4.2]{ns_flo_17} describes a process to produce a Jacobian equivalent reduced QwP from a general QwP in many cases.

\begin{example}
    Consider the ``double'' $Q$ of the quiver from Fig.~\ref{fig:mut_ex}(a). This is not approachable by the standard mutation machinery of \S\ref{sec:qui_mut} since it has (many) $2$-cycles. We equip it with the potential
    $$W=x_{01}x_{14}x_{40}+x_{04}x_{41}x_{10}+x_{02}x_{23}x_{30}+x_{03}x_{32}x_{20}+x_{10}x_{03}x_{30}x_{01}+x_{20}x_{04}x_{40}x_{02}$$
    where $x_{ij}$ denotes the arrow in $Q$ from $i$ to $j$. We show the quiver $(\wt{Q},\wt{W})$ produced by the first steps of mutation in Fig.~\ref{fig:mut_qwp_ex}(b) with original arrows grayed out, and the final mutated quiver $\mu_0(Q,W)$ in Fig.~\ref{fig:mut_qwp_ex}(c).

    \begin{figure}[h]
    \centering
\begin{tikzpicture}[scale=0.8]
\node (a) at (0,2){$1$};
\node (o) at (0,0){$0$};
\node (b) at (2,0){$2$};
\node (c) at (0,-2){$3$};
\node (d) at (-2,0){$4$};

\draw[->] (a) [out=255,in=105] to (o);
\draw[<-] (a) [out=285,in=75] to (o);
\draw[->] (o) [out=15,in=165] to (b);
\draw[<-] (o) [out=345,in=195] to (b);
\draw[->] (o) [in=15,out=165] to (d);
\draw[<-] (o) [in=345,out=195] to (d);
\draw[->] (c) [in=255,out=105] to (o);
\draw[<-] (c) [in=285,out=75] to (o);
\draw[->] (d) [out=35,in=235] to (a);
\draw[<-] (d) [out=55,in=215] to (a);
\draw[->] (b) [in=35,out=235] to (c);
\draw[<-] (b) [in=55,out=215] to (c);

\node (a) at (6,2){$\bullet$};
\node (o) at (6,0){$\star$};
\node (b) at (8,0){$\bullet$};
\node (c) at (6,-2){$\bullet$};
\node (d) at (4,0){$\bullet$};

\draw[->,color=gray] (a) [out=255,in=105] to (o);
\draw[<-,color=gray] (a) [out=285,in=75] to (o);
\draw[->,color=gray] (o) [out=15,in=165] to (b);
\draw[<-,color=gray] (o) [out=345,in=195] to (b);
\draw[->,color=gray] (o) [in=15,out=165] to (d);
\draw[<-,color=gray] (o) [in=345,out=195] to (d);
\draw[->,color=gray] (c) [in=255,out=105] to (o);
\draw[<-,color=gray] (c) [in=285,out=75] to (o);
\draw[->,color=gray] (d) [out=35,in=235] to (a);
\draw[<-,color=gray] (d) [out=55,in=215] to (a);
\draw[->,color=gray] (b) [in=35,out=235] to (c);
\draw[<-,color=gray] (b) [in=55,out=215] to (c);

\draw[->] (a) [out=305,in=145] to (b);
\draw[<-] (a) [out=325,in=125] to (b);
\draw[->] (d) [out=305,in=145] to (c);
\draw[<-] (d) [out=325,in=125] to (c);
\draw[<-] (a) [out=250,in=110] to (c);
\draw[->] (a) [out=290,in=70] to (c);
\draw[->] (a) [out=170,in=100] to (d);
\draw[<-] (a) [out=190,in=80] to (d);
\draw[->] (b) [out=260,in=10] to (c);
\draw[<-] (b) [out=280,in=350] to (c);
\draw[<-] (b) [out=160,in=20] to (d);
\draw[->] (b) [out=200,in=340] to (d);

\node (a) at (12,2){$\bullet$};
\node (o) at (12,0){$\star$};
\node (b) at (14,0){$\bullet$};
\node (c) at (12,-2){$\bullet$};
\node (d) at (10,0){$\bullet$};

\draw[<-] (a) [out=255,in=105] to (o);
\draw[->] (a) [out=285,in=75] to (o);
\draw[<-] (o) [out=15,in=165] to (b);
\draw[->] (o) [out=345,in=195] to (b);
\draw[<-] (o) [in=15,out=165] to (d);
\draw[->] (o) [in=345,out=195] to (d);
\draw[<-] (c) [in=255,out=105] to (o);
\draw[->] (c) [in=285,out=75] to (o);
\draw[<-] (d) [out=325,in=125] to (c);
\draw[->] (d) [out=305,in=145] to (c);
\draw[->] (b) [in=305,out=145] to (a);
\draw[<-] (b) [in=325,out=125] to (a);

\node at (0,-3){(a)};
\node at (6,-3){(b)};
\node at (12,-3){(c)};
\end{tikzpicture}
    \caption{Example of QwP mutation}
    \label{fig:mut_qwp_ex}
\end{figure}
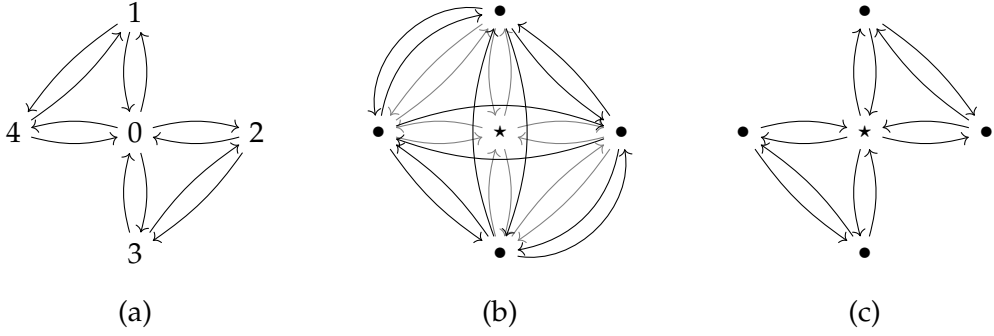

The intermediate potential is given by
\begin{align*}
\wt{W}&=y_{41}x_{41}^*+y_{14}x_{14}^*+y_{32}x_{32}^*+y_{23}x_{23}^*+y_{13}y_{31}+y_{24}y_{42} \\
&+x_{01}^*x_{20}^*y_{21}+x_{02}^*x_{10}^*y_{12}+x_{02}^*x_{30}^*y_{32}+x_{03}^*x_{20}^*y_{23}+x_{03}^*x_{40}^*y_{43}+x_{04}^*x_{30}^*y_{34} \\
&+x_{04}^*x_{10}^*y_{14}+x_{01}^*x_{40}^*y_{41}+x_{01}^*x_{30}^*y_{31}+x_{03}^*x_{10}^*y_{13}+x_{02}^*x_{40}^*y_{42}+x_{04}^*x_{20}^*y_{24}
\end{align*}
where we denote $[x_{ij}x_{jk}]=y_{ik}$ for brevity. By taking partial derivatives we find generators of $\partial\wt{W}$, hence relations in $P(\wt{Q},\wt{W})$, including
$$y_{13}+x_{01}^*x_{30}^*\quad y_{31}+x_{03}^*x_{10}^*\quad y_{42}+x_{04}^*x_{20}^*\quad y_{24}+x_{02}^*x_{40}^*$$
These relations allow us to eliminate the arrows $y_{13},y_{31},y_{24},y_{42}$ while maintaining the same Jacobian algebra up to isomorphism. The potential in this new algebra after the elimination step takes the form
\begin{align*}
&y_{41}x_{41}^*+y_{14}x_{14}^*+y_{32}x_{32}^*+y_{23}x_{23}^*+x_{01}^*x_{30}^*x_{03}^*x_{10}^*+x_{02}^*x_{40}^*x_{04}^*x_{20}^* \\
&+x_{01}^*x_{20}^*y_{21}+x_{02}^*x_{10}^*y_{12}+x_{02}^*x_{30}^*y_{32}+x_{03}^*x_{20}^*y_{23}+x_{03}^*x_{40}^*y_{43}+x_{04}^*x_{30}^*y_{34} \\
&+x_{04}^*x_{10}^*y_{14}+x_{01}^*x_{40}^*y_{41}
\end{align*}
The relations
$$x_{41}^*+x_{01}^*x_{40}^*\quad x_{14}^*+x_{04}^*x_{10}^*\quad x_{32}^*+x_{02}^*x_{30}^*\quad x_{23}^*+x_{03}^*x_{20}^*$$
allow us to remove the arrows $x_{41}^*,x_{14}^*,x_{32}^*,x_{23}^*,y_{14},y_{41},y_{23},y_{32}$ and, after cancellation, produces the final potential
$$W=x_{01}^*x_{30}^*x_{03}^*x_{10}^*+x_{02}^*x_{40}^*x_{04}^*x_{20}^*+x_{01}^*x_{20}^*y_{21}+x_{02}^*x_{10}^*y_{12}+x_{03}^*x_{40}^*y_{43}+x_{04}^*x_{30}^*y_{34}$$
\end{example}

Observe that it may not be possible to find a Jacobian-equivalent reduced QwP in general and so this mutation operation is only partially defined. We will later use computational methods to verify that the QwPs we consider are appropriately mutable.

\subsection{Abelian McKay correspondence \& toric geometry} \label{sec:abelian}

The McKay correspondence seeks to relate, through various media, the representation theory of a finite subgroup $G\subseteq\on{SL}_n(\C)$ with the geometry of the quotient singularity $\C^n/G=:\on{Spec}\C[x_1,\dots,x_n]^G$; that is, the spectrum of the ring of $G$-invariants. In practice we often pass to a minimal (or crepant) resolution of $\C^n/G$ when one exists or consider the $G$-equivariant geometry of $\C^n$ (which often coincides \cite{bkr_mck_01}). The story largely began with beautiful phenomenology in two dimensions where there is a bijection (with some extra structure) between components of the exceptional fibre of the unique minimal resolution of $\C^2/G$ and nontrivial irreducible representations of $G$. 

We will henceforth focus on the \emph{three dimensional} case and will further restrict our attention to the case whem $G$ \emph{is abelian}. In this setting the singularity $\C^3/G$ and its crepant resolutions are toric $3$-folds, enabling one to use combinatorial methods to study them. 

We briefly recall the setup for toric geometry and fix notation that we will use throughout the paper largely following \cite{cls_tor_11}. We write $G=\frac{1}{r}(a,b,c)$ to mean that $G\cong\Z/r$ is generated by the matrix
$$g=\mat{ccc}
\eps^a \\
& \eps^b \\
& & \eps^c\tam$$
where $\eps$ is a primitive $r$th root of unity.

Let $N=\Z^3$ and $N'=\Z^3+\Z\cdot(\tfrac{a}{r},\tfrac{b}{r},\tfrac{c}{r})$. Let $\sigma$ (resp.~$\sigma'$) denote the cone in $N_\R$ (resp.~in $N'_\R$) generated by the standard basis vectors $e_1=(1,0,0),e_2=(0,1,0),e_3=(0,0,1)$. The toric variety $U_\sigma$ associated to $\sigma$ is $\C^3$, and $U_{\sigma'}$ to $\sigma'$ is $\C^3/G$ with the natural map $N\subseteq N'$ inducing the quotient $\C^3\to\C^3/G$. Denote by $\Delta_1$ the slice $\sigma'\cap(e^1+e^2+e^3=1)$, where $e^i$ are the dual basis of $N^\vee_\R$ to $e_i$. We call $\Delta_1$ the \emph{junior simplex}. With this setup a crepant resolution of $\C^3/G$ corresponds to a triangulation of $\Delta_1$ with vertices in $N'$ such that each triangle is unimodal, meaning that its vertices form a $\Z$-basis of $N'$.

Let $\pi\colon Y\to\C^3/G$ be a crepant resolution and let $\mathcal{T}$ be the corresponding triangulation. We have the following correspondence between the geometry of $Y$ and the combinatorics of $\mathcal{T}$:
\begin{itemize}
\item torus-invariant exceptional curves in $Y\longleftrightarrow$ edges in $\mathcal{T}$
\item torus-invariant exceptional divisors in $Y\longleftrightarrow$ vertices in $\mathcal{T}$
\item torus-invariant compact exceptional divisors in $Y\longleftrightarrow$ interior vertices in $\mathcal{T}$
\item torus-fixed points in $Y\longleftrightarrow$ triangles in $\mathcal{T}$
\end{itemize}

A standard move in birational geometry to produce different minimal resolutions is the \emph{flop} in a suitable curve. The combinatorial version of the flop in a torus-invariant exceptional curve $C$ represented by an edge in a triangulation is \emph{flipping} the edge as shown in Fig.~\ref{fig:flip}.

\begin{figure}[h]
\begin{center}
\begin{tikzpicture}[scale=1.4]
\foreach \i in {1,...,2}
{
\node (a\i) at (0,\i){\tiny$\bullet$};
}
\foreach \i in {1,...,2}
{
\node (b\i) at (1,\i+0.5){\tiny$\bullet$};
}

\draw[line width = 1.2pt](a1.center) to (b2.center);
\draw (a1.center) to (a2.center);
\draw (a1.center) to (b1.center);
\draw (a2.center) to (b2.center);
\draw (b1.center) to (b2.center);

\foreach \i in {1,...,2}
{
\node (a\i) at (2,\i){\tiny$\bullet$};
}
\foreach \i in {1,...,2}
{
\node (b\i) at (3,\i+0.5){\tiny$\bullet$};
}

\draw[line width = 1.2pt](a2.center) to (b1.center);
\draw (a1.center) to (a2.center);
\draw (a1.center) to (b1.center);
\draw (a2.center) to (b2.center);
\draw (b1.center) to (b2.center);
\end{tikzpicture}
\end{center}
\caption{Flipping an edge}
\label{fig:flip}
\end{figure}
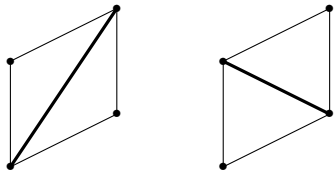

As we will explore further below, there are combinatorial constraints on whether this is possible; this translates in geometry to the normal bundle of the relevant curve being of suitable type \cite{pin_fac_83}.

There is a distinguished minimal resolution of $\C^3/G$: the $G$-Hilbert scheme $G\hilb\C^3$ studied in this context first by Nakamura \cite{nak_hil_01}. This is the moduli space of $G$-clusters; essentially scheme-theoretic $G$-orbits. Craw--Reid \cite{cr_how_02} describe a procedure for computing the triangulation for $G\hilb\C^3$, which works by dividing the junior simplex into a collection of `regular triangles' \cite[\S1.2]{cr_how_02} and then applying a standard subdivision to each of these pieces. We integrated this algorithm into our computational tools, which we flesh out below.

\subsection{Curve quivers} \label{sec:cq}

In many cases, given a resolution (or more generally a birational map) $\pi\colon Y\to X$ between $3$-folds one can construct a quiver $Q_\pi$ associated to $\pi$. In the case of resolutions we often denote $Q_\pi=Q_Y$ so long as context is clear. The construction is as follows:

\begin{itemize}
    \item $Q_\pi$ has a node for each irreducible curve in the intersection of two irreducible exceptional divisors. We will assume that all such curves are \emph{rational} for the following.
    \item Two vertices $i$ and $k$ corresponding to distinct curves $C_i$ and $C_j$ have a $2$-cycle between them if $C_i\cap C_j\not=\emptyset$. Otherwise, $i$ and $j$ have no arrows between them.
    \item Let $i$ be the node corresponding to a curve $C_i\subseteq D_1\cap D_2$ for exceptional divisors $D_1,D_2$. The number of loops at $i$ is the number $n$ determined by the degrees of the normal bundles $\mathcal{N}_{D_k/C_i}$ of $C$ inside $D_1$ and $D_2$. Namely, since $C_i$ is rational there is $n\in\Z$ such that $\mathcal{N}_{D_1/C_i}=\mO_{C_i}(-n)$ and $\mathcal{N}_{D_2/C_i}=\mO_{C_i}(n-2)$.
\end{itemize}

In the case where $\pi$ is toric or a resolution of a polyhedral singularity (where the exceptional fibre is one-dimensional) each of the curves $C_i$ are rational and thus the curve quiver is well-defined in this context.

\begin{example}
    Let $G=\frac{1}{6}(1,2,3)$. The toric picture of the crepant resolution $\pi\colon G\hilb\C^3\to\C^3/G$ is shown in Fig.~\ref{fig:1/6b}(a). Using the dictionary from \S\ref{sec:abelian} the curve quiver of $\pi$ can be computed to be as in Fig.\ref{fig:1/6b}(b).
    
    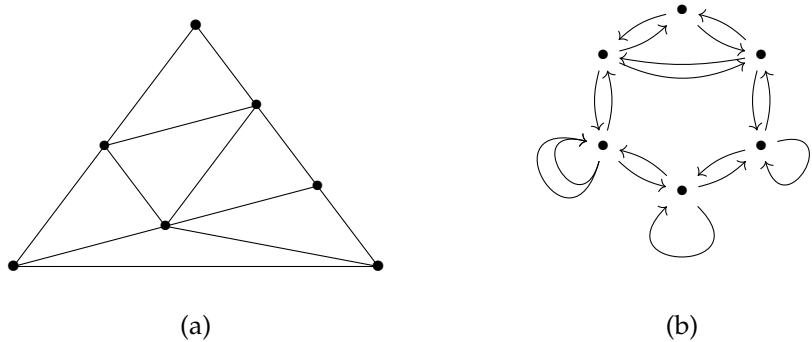
\begin{figure}[h]
    \begin{center}
    \begin{tikzpicture}[scale=0.8]

\node(e1) at (0-8,2){$\bullet$};
\node(e2) at (3-8,-2){$\bullet$};
\node(e3) at (-3-8,-2){$\bullet$};

\small

\node(123) at (-0.5-8,-4/3){$\bullet$};
\node(240) at (2-8,2-8/3){$\bullet$};
\node(420) at (1-8,2-4/3){$\bullet$};
\node(303) at (-1.5-8,0){$\bullet$};

\draw[-] (e2.center) to (123.center);
\draw[-] (e3.center) to (123.center) to (240.center);
\draw[-] (303.center) to (420.center);
\draw[-] (123.center) to (303.center);
\draw[-] (123.center) to (420.center);

\draw[-] (e1.center) to (e2.center) to (e3.center) to (e1.center);

\begin{scope}[scale=1.5,yshift=0.5cm]
	\node (a1) at ({sin(0)},{cos(0)}) {$\bullet$};
	\node (a2) at ({sin(60)},{cos(60)}) {$\bullet$};
	\node (a3) at ({sin(120)},{cos(120)}) {$\bullet$};
	\node (a4) at ({sin(180)},{cos(180)}) {$\bullet$};
	\node (a5) at ({sin(240)},{cos(240)}) {$\bullet$};
	\node (a6) at ({sin(300)},{cos(300)}) {$\bullet$};
	\draw[->] (a1) to [bend right=15] (a2);
	\draw[->] (a2) to [bend right=15] (a3);
	\draw[->] (a3) to [bend right=15] (a4);
	\draw[->] (a4) to [bend right=15] (a5);
	\draw[->] (a5) to [bend right=15] (a6);
	\draw[->] (a6) to [bend right=15] (a1);
	\draw[->] (a2) to [bend right=15] (a1);
	\draw[->] (a3) to [bend right=15] (a2);
	\draw[->] (a4) to [bend right=15] (a3);
	\draw[->] (a5) to [bend right=15] (a4);
	\draw[->] (a6) to [bend right=15] (a5);
	\draw[->] (a1) to [bend right=15] (a6);
	\draw[->] (a2) to [bend right=-10] (a6);
	\draw[->] (a6) to [bend right=25] (a2);
	
	\draw[->] (a3) to [out=-120+90+45,in=-120+90-45,looseness=8] (a3);
	\draw[->] (a4) to [out=-180+90+45,in=-180+90-45,looseness=8] (a4);
	\draw[->] (a5) to [out=-240+90+45,in=-240+90-45,looseness=8] (a5);
	\draw[->] (a5) to [out=-240+90+45,in=-240+90-45,looseness=12] (a5);
	\end{scope}
 
\node (la) at (-8,-3){(a)};
\node (la) at (0,-3){(b)};

 \end{tikzpicture}
 \end{center}
\caption{$G$-Hilb and its curve quiver for $\frac{1}{6}(1,2,3)$}
\label{fig:1/6b}
\end{figure}
\end{example}

In general for toric crepant $3$-fold resolutions one can read the curve quiver from the associated triangulation. Each edge in the triangulation corresponds to a rational curve obtained as the intersection of two exceptional divisors (i.e.~a node in the curve quiver), intersections between curves (i.e.~$2$-cycles in the curve quiver) are described by the two corresponding edges partially bounding a triangle, and the normal bundles (i.e.~loops in the curve quiver) can be read off from suitable linear relations between vertices in the triangulation \cite[\S6.2]{cls_tor_11}.

In more detail, for a node $i$ corresponding to an edge $e$ of a triangulation, the number of loops at $i$ is based on the linear relations between the edges of the quadrilateral for which $e$ is the diagonal. If the vertices of the quadrilateral are enumerated $v_1,v_2,v_3,v_4$ with $e$ linking $v_1$ and $v_3$ then there is a relation 
$$v_1 + v_3 = \alpha v_1 + \beta v_2$$
The constraint that all vertices live in the simplex at height one forces $\alpha+\beta=2$. There are three possibilities:
$$\{\alpha,\beta\}\quad=\quad\{1,1\}\quad\{0,2\}\quad\{1-n,n+1\}\text{ for }n\geq3$$
We depict these three possibilities in Fig.~\ref{fig:quad}(a)-(c). In each case, the normal bundles of the curve $C$ corresponding to $e$ inside the exceptional divisors corresponding to $v_1$ and $v_3$ are $\mO_C(-\alpha)$ and $\mO_C(-\beta)$ respectively.

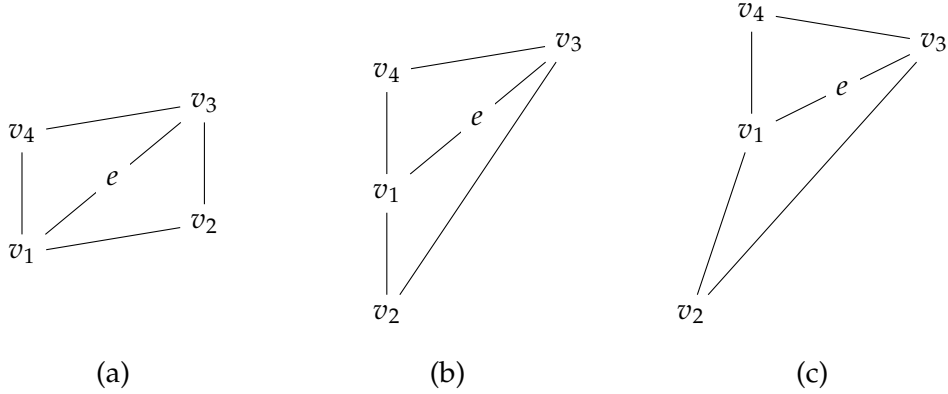
\begin{figure}[h]
    \centering
    \begin{tikzpicture}[scale=0.8]
    \node (a) at (-1,2){$v_4$};
    \node (b) at (-1,0){$v_1$};
    \node (c) at (2,.5){$v_2$};
    \node (d) at (2,2.5) {$v_3$};
    \draw (a) -- (b) -- (c) -- (d) -- (a);
    \draw (b) -- node[fill=white] {$e$} (d);

    \node (a) at (5,3){$v_4$};
    \node (b) at (5,1){$v_1$};
    \node (c) at (5,-1){$v_2$};
    \node (d) at (8,3.5) {$v_3$};
    \draw (a) -- (b) -- (c) -- (d) -- (a);
    \draw (b) -- node[fill=white] {$e$} (d);

    \node (a) at (11,4){$v_4$};
    \node (b) at (11,2){$v_1$};
    \node (c) at (10,-1){$v_2$};
    \node (d) at (14,3.5) {$v_3$};
    \draw (a) -- (b) -- (c) -- (d) -- (a);
    \draw (b) -- node[fill=white] {$e$} (d);

    \node (l1) at (0.5,-2){(a)};
    \node (l2) at (6,-2){(b)};
    \node (l3) at (12,-2){(c)};
    \end{tikzpicture}
\caption{Relations from quadrilaterals}
\label{fig:quad}
\end{figure}

Thus, (a) is the case where $n=0$ and the corresponding node has no loops, (b) is where $n=1$ and the node has one loop, and (c) is where $n\geq 2$ and the node has two or more loops. In case (a) the edge $e$ can be flipped to obtain a new triangulation; corresponding to the geometric fact that only curves with normal bundle of type $\mO(-1,-1)$ can be flopped in the toric setting.

\vspace{1em}

We conclude by noting that we need not restrict ourselves to resolutions of abelian quotient singularities; for the triangulation and curve quiver machinery to apply, we could broaden our scope to include all crepant resolutions of Gorenstein toric singularities, which have recently been studied in many analogous ways to quotient singularities in terms of dimer models \cite{cht_com_20,iu_dim_16}.

\section{The space of potentials}

Fixing a quiver, we consider various ways of understanding the set of potentials on that quiver. This set, namely, the algebra $\kqcyc$ is naturally a countably infinite dimensional $k$-vector space but this description is inconvenient for most purposes. We introduce various stratifications of $\kqcyc$ that better serve our interests.

\subsection{Invariants of potentials}

Each of the stratifications we consider arise from numerical or combinatorial invariants of QwPs.

\begin{definition}
    Let $Q$ be a quiver. The \emph{total exchange number} $\on{ten}(W)$ of a potential $W\in\kqcyc$ is the number of distinct QwPs mutation-equivalent to $(Q,W)$. The \emph{exchange number} $\on{en}(W)$ of $W$ is the number of distinct quivers $Q'$ such that there exists a potential $W\in kQ'_\text{cyc}$ with $(Q',W')$ mutation-equivalent to $(Q,W)$.
\end{definition}

\begin{example}
    Data on exchange number stratification for $\frac{1}{6}(1,2,3)$, calculated from constructing 1 million randomly-generated potentials and finding all distinct quivers for each.
        \begin{center}
        \begin{tabular}{c|c}
             Exchange Number & Prevalence \\
             \hline
             1 & 32.7\% \\
             2 & 20.4\% \\
             3 & 19.7\% \\
             4 & 20.8\% \\
             5 & 5.0\% \\
             6 & 1.3\% \\
        \end{tabular}
    \end{center}
\end{example}

As we will discuss at a later point, the exchange number is slightly too coarse a measurement for our purposes. However, it will still prove to be useful to refine our searches for interesting potentials. We next introduce the more refined invariant that will be more fruitful for us to consider.

\begin{definition}
    Let $Q$ be a quiver. We define the \emph{quiver set} $\on{qs}(W)$ of a potential $W$ on $Q$ to be the set of all quivers $Q'$ admitting a potential $W'$ such that $(Q',W')$ is mutation-equivalent to $(Q,W)$.
\end{definition}

Given a set $S$ of quivers including $Q$, we obtain a subset of the space of potentials $\kqcyc$ consisting of all potentials $W$ with $\on{qs}(W)=S$. Of course, most of these subsets will be empty. Observe that
$$\{W\in\kqcyc:\on{qs}(W)=S\}\subseteq\{W\in\kqcyc:\on{en}(W)=|S|\}$$

\begin{example}
    We provide some data on the quiver set stratification for $G=\frac{1}{6}(1,2,3)$. We obtained this data by constructing one million randomly-generated potentials with up to 30 terms and finding the quiver set for each. We then group the quiver sets by checking whether the nodes have matching degrees as an approximation for determining whether they are isomorphic; a true check for isomorphism is infeasible to compute in a reasonable amount of time.

\vspace{1em}

    In total, we found 3026 distinct sets. The five most frequently represented sets had the following numbers of potentials associated with them. We also include the exchange number of each stratum; that is, the size of the respective quiver sets.
    \begin{center}
        \begin{tabular}{c | c}
             Number of Potentials & Exchange Number \\
             \hline
             325963 & 1 \\
             56740 & 2 \\
             46419 & 2 \\
             19789 & 3 \\
             18668 & 2
        \end{tabular}
    \end{center}
    Only a small number of quiver sets occur with this many potentials. 51\% (1542) of the quiver sets appeared 10 or fewer times, and 16\% (489) appeared only once in our testing.

\vspace{1em}
    
    We also examined the distribution of quiver sets by exchange number.
    \begin{center}
        \begin{tabular}{c | c}
             Exchange Number & Number of Distinct Quiver Sets Found \\
             \hline
             1 & 1 \\
             2 & 19 \\
             3 & 139 \\
             4 & 580 \\
             5 & 1105 \\
             6 & 1182
        \end{tabular}
    \end{center}
\end{example}

In our applications we have a set of quivers already in mind -- curve quivers of crepant resolutions of singularities -- and so it is an interesting open question as to whether the corresponding subset of potentials in empty and, if not, how large it is. In the next subsection we discuss how to find a distinguished representative of this subset when it is nonempty.

\vspace{1em}

The final and finest invariant we draw on is the \emph{exchange graph} of a potential.

\begin{definition}
    Let $Q$ be a quiver. We define the \emph{exchange graph} $\on{eg}(W)$ of a potential $W$ on $Q$ to be the following labelled graph. The underlying graph $\Gamma(W)$ has a node for each QwP mutation-equivalent to $(Q,W)$ and an edge between two nodes if there is a (one-step) mutation between the corresponding QwPs. Each node is labelled by the corresponding QwP, which we encode as a function $q\colon \Gamma(W)_0\to\mathcal{FQ}$, where $\mathcal{FQ}$ is the set of finite quivers.
\end{definition}

When $W=0$ and $Q$ has no loops or $2$-cycles this is simply the exchange graph of the cluster algebra corresponding to $Q$ labelled by the appropriate quiver mutation at each step. This clearly gives rise to the finest stratification we consider here: if $(\Gamma,q)$ is a labelled graph as above we have
$$\{W\in\kqcyc:\on{eg}(W)=(\Gamma,q)\}\subseteq\{W\in\kqcyc:\on{qs}(W)=q(\Gamma_0)\}$$

\begin{example}
    Data on exchange graph stratification for $\frac{1}{6}(1,2,3)$. (i.e.~how many of the quivers with the right quiver set actually have the right exchange graph? my guess is all or most).
\end{example}

\subsection{Minimal potentials}

Let $Q$ be a quiver. We say that the \emph{support} $\on{supp}(W)$ of a potential $W$ on $Q$ is the set of cycles appearing as monomials in $W$.

\begin{definition}
    Let $Q$ be a quiver. We say that a potential $W$ on $Q$ is \emph{minimal} if there is no potential $W'$ on $Q$ with $\on{supp}(W')\subsetneq \on{supp}(W)$ such that $\on{qs}(W)=\on{qs}(W')$.
\end{definition}

\begin{lemma} \label{lem:deg_3}
    Let $(Q,W)$ be a quiver with potential whose exchange graph is identified with the flip graph of a triangulation. Let $i\in Q_0$ be a loopless node and assume that an arrow $x_{jk}$ is removed by mutation at $i$. If $j\not=k$ then $W$ contains
    $$x_{ijk}:=x_{ij}x_{jk}x_{ki}+x_{ik}x_{kj}x_{ji}$$
    If $j=k$ then $W$ contains
    $$x_{ijj}:=x_{ij}x_{jj}x_{ji}$$
\end{lemma}

\begin{proof}
    Let $m$ be a term causing $x_{jk}$ to be removed. There are ten possibilities for $m$, most of which are outlawed by the condition that $i$ is loopless and by analysing their behaviour under mutation to see if edges are removed or not. From this list only $x_{ijj}$ removes $x_{jk}$ when $j=k$. The two terms that remove an edge from $j$ to $k$ are
    $$x_{ij}x_{jk}x_{ki}\qquad x_{ijk}':=x_{ji}x_{ik}x_{ki}x_{ij}$$
    The latter does not in fact remove any of the original edges from $Q$ under mutation, hence $W$ must contain the former. Let $\mu_1(Q,W)=(Q',W')$. Since $Q'$ is the quiver of a triangulation we must also remove $x_{kj}$ in this case and so $W$ must also contain $x_{ik}x_{kj}x_{ji}$.
\end{proof}

Given a triangulation $\mathcal{T}$ as in \S\ref{sec:abelian} we can define a labelled graph $(\Gamma,q)$ called the \emph{labelled flip graph} of $\mathcal{T}$ via:
\begin{itemize}
    \item the nodes of $\Gamma$ correspond to the triangulations related to $\mathcal{T}$ by a series of flips,
    \item two nodes are connected by an edge if there is a single flip relating them,
    \item the value of a node under $q$ is the curve quiver of the corresponding triangulation (or toric resolution).
\end{itemize}

\begin{prop} \label{prop:minimal}
    Let $Q$ be the curve quiver of a triangulation (or toric resolution). There is at most one minimal potential (up to generic choice of coefficients) on $Q$ whose exchange graph is identified with the labelled flip graph of the triangulation.
\end{prop}

\begin{proof}
Let $Q$ be the quiver of a triangulation, let $W_1$ and $W_2$ be minimal potentials generating the exchange graph of the triangulation. Let $m$ be a monomial in $W_1$ not in $W_2$. As $W_1$ is minimal there has to be a sequence $\mu_I$ of mutations where $m$ causes cancellation at the next step, mutation in $i_n$. Let $(Q',W_i')=\mu_I(Q,W_i)$. Then the image of $m$ under $\mu_I$ is of the form
$$x_{ij}x_{jk}x_{ki}$$
By Lem.~\ref{lem:deg_3} this term is also in $W_2'$ because $Q'$ is the quiver of a triangulation. Inverting the sequence of mutations yields that $m$ is also a term in $W_2$, contradiction.
\end{proof}

\section{The \texttt{qwp\_mutations} application} \label{sec:app}

To explore the properties of quivers with potentials we developed a web-based tool for viewing and manipulating them. The \texttt{qwp\_mutations} application displays a visualization of a given quiver and has an interface that allows for constructing an arbitrary quiver. It also contains an implementation of the mutation process described in \S\ref{sec:qwp_mut}.

\subsection{Triangulation}

As discussed above, for a finite subgroup $G=\frac{1}{r}(a,b,c)$ of $\on{SL}_3(\C)$ the triangulation for $G\hilb\C^3$ can be calculated directly using the Craw--Reid procedure \cite{cr_how_02}. We implemented this algorithmically in the \texttt{qwp\_mutations} application.

\begin{remark} \label{rem:moc}
There is a special case in the Craw--Reid procedure referred to as the ``meeting of champions'' \cite{cr_how_02}. This produces a set of internal rays that cannot be extended to create a partition of the simplex face into regular triangles using the Jung--Hirzebruch continued fraction method, and such cases are considered as a class of problems for which the current implementation is not equipped. The \texttt{triangulation} routine provided in the \texttt{qwp\_mutations} application does not produce a valid triangulation for such cases.
\end{remark}

\subsection{Determining the potential for a quiver set}

The \texttt{quiver\_mutations} application also implements a heuristics-based routine that accepts a set of quivers generated from the triangulation tool, and attempts to find a potential that generates the expected set of quivers when paired with the first quiver in the set.

\vspace{1em}

Given one quiver, a node at which to mutate, and the expected resulting quiver we can compare the expected to the actual result of the mutation using the current potential and determine which edges should be removed. We can then construct a potential term that removes each edge; by Lem.~\ref{lem:deg_3} there is only a single set of potential terms that can be added to remove a given edge. We also rely on the fact that mutating at a node only affects terms involving nodes that are separated from the mutated node by a single arrow. Therefore, given two quivers, we can find a list of nodes that could possibly be mutated at to cause the first quiver to be mutated into the second quiver.

\begin{example} We include the curve quivers for $G=\frac{1}{6}(1,2,3)$ in Fig.~\ref{fig:1/6c} arranged by flips in the corresponding triangulations. It is clear from this figure that there is only one possible mutation connecting each pair of adjacent quivers and no possible mutation connecting non-adjacent quivers. For instance, to move from the left-most to the second-from-left quiver the only option is to mutate at the top-most node since it is the only loop-less node. To move to the next quiver to the right, the only option is to mutate at the top-right node, since loops are induced or removed at all adjacent nodes. To instead move downward to the bottom-most quiver the only option is to mutate at the top-left node, either by considering added/removed loops or by noting that the only loop-less node in the mutated quiver is the top-right.

\begin{figure}[h]
    {\begin{center}
\begin{tikzpicture}[scale=0.85]
	\node (a1) at ({sin(0)},{cos(0)}) {$\bullet$};
	\node (a2) at ({sin(60)},{cos(60)}) {$\bullet$};
	\node (a3) at ({sin(120)},{cos(120)}) {$\bullet$};
	\node (a4) at ({sin(180)},{cos(180)}) {$\bullet$};
	\node (a5) at ({sin(240)},{cos(240)}) {$\bullet$};
	\node (a6) at ({sin(300)},{cos(300)}) {$\bullet$};
	\draw[->] (a1) to [bend right=15] (a2);
	\draw[->] (a2) to [bend right=15] (a3);
	\draw[->] (a3) to [bend right=15] (a4);
	\draw[->] (a4) to [bend right=15] (a5);
	\draw[->] (a5) to [bend right=15] (a6);
	\draw[->] (a6) to [bend right=15] (a1);
	\draw[->] (a2) to [bend right=15] (a1);
	\draw[->] (a3) to [bend right=15] (a2);
	\draw[->] (a4) to [bend right=15] (a3);
	\draw[->] (a5) to [bend right=15] (a4);
	\draw[->] (a6) to [bend right=15] (a5);
	\draw[->] (a1) to [bend right=15] (a6);
	
	\draw[->] (a2) to [out=-60+90+45, in=-60+90-45,looseness=8] (a2);
	\draw[->] (a3) to [out=-120+90+45,in=-120+90-45,looseness=8] (a3);
	\draw[->] (a4) to [out=-180+90+45,in=-180+90-45,looseness=8] (a4);
	\draw[->] (a5) to [out=-240+90+45,in=-240+90-45,looseness=8] (a5);
	\draw[->] (a5) to [out=-240+90+45,in=-240+90-45,looseness=12] (a5);
	\draw[->] (a6) to [out=-300+90+45,in=-300+90-45,looseness=8] (a6);
	
	\begin{scope}[xshift=4cm]
	\node (a1) at ({sin(0)},{cos(0)}) {$\bullet$};
	\node (a2) at ({sin(60)},{cos(60)}) {$\bullet$};
	\node (a3) at ({sin(120)},{cos(120)}) {$\bullet$};
	\node (a4) at ({sin(180)},{cos(180)}) {$\bullet$};
	\node (a5) at ({sin(240)},{cos(240)}) {$\bullet$};
	\node (a6) at ({sin(300)},{cos(300)}) {$\bullet$};
	\draw[->] (a1) to [bend right=15] (a2);
	\draw[->] (a2) to [bend right=15] (a3);
	\draw[->] (a3) to [bend right=15] (a4);
	\draw[->] (a4) to [bend right=15] (a5);
	\draw[->] (a5) to [bend right=15] (a6);
	\draw[->] (a6) to [bend right=15] (a1);
	\draw[->] (a2) to [bend right=15] (a1);
	\draw[->] (a3) to [bend right=15] (a2);
	\draw[->] (a4) to [bend right=15] (a3);
	\draw[->] (a5) to [bend right=15] (a4);
	\draw[->] (a6) to [bend right=15] (a5);
	\draw[->] (a1) to [bend right=15] (a6);
	\draw[->] (a2) to [bend right=-10] (a6);
	\draw[->] (a6) to [bend right=25] (a2);
	
	\draw[->] (a3) to [out=-120+90+45,in=-120+90-45,looseness=8] (a3);
	\draw[->] (a4) to [out=-180+90+45,in=-180+90-45,looseness=8] (a4);
	\draw[->] (a5) to [out=-240+90+45,in=-240+90-45,looseness=8] (a5);
	\draw[->] (a5) to [out=-240+90+45,in=-240+90-45,looseness=12] (a5);
	\end{scope}
	
	\begin{scope}[xshift=4cm, yshift=-4.5cm]
	\node (a1) at ({sin(0)},{cos(0)}) {$\bullet$};
	\node (a2) at ({sin(60)},{cos(60)}) {$\bullet$};
	\node (a3) at ({sin(120)},{cos(120)}) {$\bullet$};
	\node (a4) at ({sin(180)},{cos(180)}) {$\bullet$};
	\node (a5) at ({sin(240)},{cos(240)}) {$\bullet$};
	\node (a6) at ({sin(300)},{cos(300)}) {$\bullet$};
	\draw[->] (a2) to [bend right=15] (a3);
	\draw[->] (a3) to [bend right=15] (a4);
	\draw[->] (a4) to [bend right=15] (a5);
	\draw[->] (a5) to [bend right=15] (a6);
	\draw[->] (a6) to [bend right=15] (a1);
	\draw[->] (a3) to [bend right=15] (a2);
	\draw[->] (a4) to [bend right=15] (a3);
	\draw[->] (a5) to [bend right=15] (a4);
	\draw[->] (a6) to [bend right=15] (a5);
	\draw[->] (a1) to [bend right=15] (a6);
	\draw[->] (a2) to [bend right=-10] (a6);
	\draw[->] (a6) to [bend right=25] (a2);
	\draw[->] (a2) to [bend right=-10] (a5);
	\draw[->] (a5) to [bend right=25] (a2);
	
	\draw[->] (a1) to [out=-0+90+45,in=-0+90-45,looseness=8] (a1);
	\draw[->] (a2) to [out=-60+90+45,in=-60+90-45,looseness=8] (a2);
	\draw[->] (a3) to [out=-120+90+45,in=-120+90-45,looseness=8] (a3);
	\draw[->] (a4) to [out=-180+90+45,in=-180+90-45,looseness=8] (a4);
	\draw[->] (a5) to [out=-240+90+45,in=-240+90-45,looseness=8] (a5);
	\draw[->] (a5) to [out=-240+90+45,in=-240+90-45,looseness=10] (a5);
	\draw[->] (a5) to [out=-240+90+45,in=-240+90-45,looseness=12] (a5);
	\end{scope}
	
	\begin{scope}[xshift=8cm]
	\node (a1) at ({sin(0)},{cos(0)}) {$\bullet$};
	\node (a2) at ({sin(60)},{cos(60)}) {$\bullet$};
	\node (a3) at ({sin(120)},{cos(120)}) {$\bullet$};
	\node (a4) at ({sin(180)},{cos(180)}) {$\bullet$};
	\node (a5) at ({sin(240)},{cos(240)}) {$\bullet$};
	\node (a6) at ({sin(300)},{cos(300)}) {$\bullet$};
	\draw[->] (a1) to [bend right=15] (a2);
	\draw[->] (a2) to [bend right=15] (a3);
	\draw[->] (a3) to [bend right=15] (a4);
	\draw[->] (a4) to [bend right=15] (a5);
	\draw[->] (a5) to [bend right=15] (a6);
	\draw[->] (a2) to [bend right=15] (a1);
	\draw[->] (a3) to [bend right=15] (a2);
	\draw[->] (a4) to [bend right=15] (a3);
	\draw[->] (a5) to [bend right=15] (a4);
	\draw[->] (a6) to [bend right=15] (a5);
	\draw[->] (a2) to [bend right=-10] (a6);
	\draw[->] (a6) to [bend right=25] (a2);
	\draw[->] (a3) to [bend right=-10] (a6);
	\draw[->] (a6) to [bend right=25] (a3);
	
	\draw[->] (a1) to [out=-0+90+45,in=-0+90-45,looseness=8] (a1);
	\draw[->] (a4) to [out=-180+90+45,in=-180+90-45,looseness=8] (a4);
	\draw[->] (a5) to [out=-240+90+45,in=-240+90-45,looseness=8] (a5);
	\draw[->] (a5) to [out=-240+90+45,in=-240+90-45,looseness=12] (a5);
	\draw[->] (a6) to [out=-300+90+45,in=-300+90-45,looseness=8] (a6);
	\end{scope}
	
	\begin{scope}[xshift=12cm]
	\node (a1) at ({sin(0)},{cos(0)}) {$\bullet$};
	\node (a2) at ({sin(60)},{cos(60)}) {$\bullet$};
	\node (a3) at ({sin(120)},{cos(120)}) {$\bullet$};
	\node (a4) at ({sin(180)},{cos(180)}) {$\bullet$};
	\node (a5) at ({sin(240)},{cos(240)}) {$\bullet$};
	\node (a6) at ({sin(300)},{cos(300)}) {$\bullet$};
	\draw[->] (a1) to [bend right=15] (a2);
	\draw[->] (a2) to [bend right=15] (a3);
	\draw[->] (a3) to [bend right=15] (a4);
	\draw[->] (a4) to [bend right=15] (a5);
	\draw[->] (a5) to [bend right=15] (a6);
	\draw[->] (a2) to [bend right=15] (a1);
	\draw[->] (a3) to [bend right=15] (a2);
	\draw[->] (a4) to [bend right=15] (a3);
	\draw[->] (a5) to [bend right=15] (a4);
	\draw[->] (a6) to [bend right=15] (a5);
	\draw[->] (a6) to [bend right=-10] (a4);
	\draw[->] (a4) to [bend right=25] (a6);
	\draw[->] (a3) to [bend right=25] (a6);
	\draw[->] (a6) to [bend right=-10] (a3);
	
	\draw[->] (a1) to [out=-0+90+45,in=-0+90-45,looseness=8] (a1);
	\draw[->] (a2) to [out=-60+90+45, in=-60+90-45,looseness=8] (a2);
	\draw[->] (a4) to [out=-180+90+45,in=-180+90-45,looseness=8] (a4);
	\draw[->] (a4) to [out=-180+90+45,in=-180+90-45,looseness=12] (a4);
	\draw[->] (a5) to [out=-240+90+45,in=-240+90-45,looseness=8] (a5);
	\draw[->] (a5) to [out=-240+90+45,in=-240+90-45,looseness=12] (a5);
	\draw[->] (a6) to [out=-300+90+45,in=-300+90-45,looseness=8] (a6);
	\draw[->] (a6) to [out=-300+90+45,in=-300+90-45,looseness=12] (a6);
	\end{scope}
	
\end{tikzpicture}
\end{center}}
\caption{Curve quivers for $\frac{1}{6}(1,2,3)$}
\label{fig:1/6c}
\end{figure}
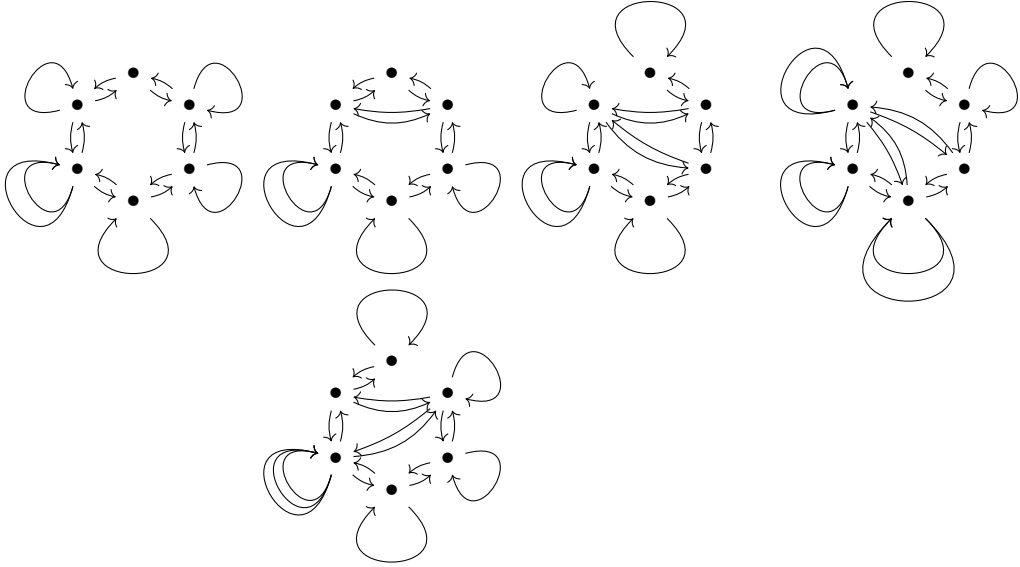

\end{example}

Using the two facts described above we iterate over all quivers and attempt to find relationships between them:

\begin{verbatim}
    mutationSequences = []
    while mutationSequences has empty terms:
        for quiver1 in quivers:
            for quiver2 in quivers:
                possible = possibleMutationNodes(quiver1, quiver2)
                if possible.length == 1
                    and canConstructPotential(quiver1, quiver2, possible[0]):
                    mutationSequences[quiver1] = mutationSequences[quiver2] + possible[0]
\end{verbatim}

Once this routine is finished we have a plausible relationship between the quivers. At this point we cannot be certain that the relationship between the quivers is correct but we will verify it later.


\vspace{1em}

Once we have found a sequence of mutations corresponding to each quiver we sort the quivers by the length of their associated mutation sequences starting with the initial quiver. For each quiver we then perform the following steps:

\begin{enumerate}
    \item Set the quiver's potential to all of the potential terms we have found so far, and mutate at each node in the quiver's mutation sequence up to the $(n-1)$th mutation. Because all of the quivers are sorted by length, the potential terms for these mutations will have already been found, and the result will be some existing quiver in the set.
    \item Then we perform the $n$th mutation, which we have not previously explored, and compare the resulting quiver to the quiver we're currently analyzing. If some potential terms were missing i the $(n-1)$th quiver, this quiver will have some extra edges.
    \item We add potential terms to the $(n-1)$th quiver that eliminate these extra edges.
    \item Then, we take the $(n-1)$th quiver and mutate it back to the base quiver. The resulting potential of that quiver is our new potential, which we use on the next quiver.
\end{enumerate}


Once we have completed this process for every quiver, the potential contains the correct terms necessary to eliminate the proper edges on every mutation. To confirm that we constructed the correct potential, we take the potential and the first quiver in the set, perform the mutations again in the sequences we identified using the complete potential, and compare the resulting quiver set to the expected set.



\subsection{Success rate} \label{sec:list} On every valid $a, b, c$ subgroup from $(1, 1, 1)$ up to $(8, 8, 8)$, our program is able to find a potential for 72 out of 95 (75\%) of the tested quiver sets. For 23 of the quivers, we fail to build a valid potential. The apparent cause of this is that we construct a potential that has a quadratic term that cannot be reduced. An example of this follows.

\vspace{1em}

In order to verify that each potential we constructed is minimal, we test removing each term in the constructed potential one at at time. We then find all possible quivers using that potential, and verify that the smaller quiver set does not produce the correct result.


\vspace{1em}

The list of triples for which we find a valid potential is:

\vspace{0.5em}
{\small \noindent
(1,1,1), (1,1,2), (1,1,3), (1,1,4), (1,1,5), (1,1,6), (1,1,7), (1,1,8), (1,2,2), (1,2,3), (1,2,4), (1,2,5), (1,2,6), (1,2,7), (1,2,8), (1,3,3), (1,3,4), (1,3,5), (1,3,6), (1,3,7), (1,4,4), (1,4,5), (1,4,6), (1,4,7), (1,4,8), (1,5,5), (1,5,6), (1,5,7), (1,6,6), (1,6,7), (1,7,7), (1,7,8), (1,8,8), (2,2,3), (2,2,5), (2,2,7), (2,3,3), (2,3,4), (2,3,5), (2,3,6), (2,3,8), (2,4,5), (2,4,7), (2,5,5), (2,5,6), (2,5,7), (2,7,7), (3,3,4), (3,3,5), (3,3,7), (3,3,8), (3,4,4), (3,4,6), (3,4,7), (3,5,5), (3,5,8), (3,7,7), (3,8,8), (4,4,5), (4,4,7), (4,5,5), (4,7,7), (5,5,6), (5,5,7), (5,5,8), (5,6,6), (5,7,7), (5,8,8), (6,6,7), (6,7,7), (7,7,8), (7,8,8).
}
\vspace{0.5em}

We also tested several other families depending on a parameter $a\in\N$. For each family we tested values of $a$ starting at 1 and stopping at the largest number for which the computation terminated in a reasonable time.


\begin{itemize}
    \item $\frac{1}{r}(1,2,a)$, $a \leq 30$: 9 successes (succeeds on $a = 1, 2, 3, 4, 5, 6, 7, 8, 10$).
    \item $\frac{1}{r}(1,a,a^2)$, $a \leq 5$: 3 successes ($a = 1, 2, 3$).
    \item $\frac{1}{r}(1,a,2a)$, $a \leq 13$: 4 successes ($a = 1, 2, 3, 4$).
\end{itemize}

These add two more positive examples not found in the list above: $\frac{1}{13}(1,2,10)$ and $\frac{1}{13}(1,3,9)$. As has been the case throughout our data-gathering, the potential search failing does not mean that there is no suitable potential but simply that our search method did not locate one.


\begin{example}
We generate the correct potential for $G = \frac{1}{11}(1,2,8)$. We  begin with the quiver in Fig.~\ref{fig:1/128_0} that we take as our starting quiver and the four other quivers in Fig.~\ref{fig:1/128_1_to_4} associated to the other resolutions of the $\frac{1}{11}(1,2,8)$ singularity that together make up the expected quiver set.

\begin{figure}[h]
    \centering
    \includegraphics[width=3in]{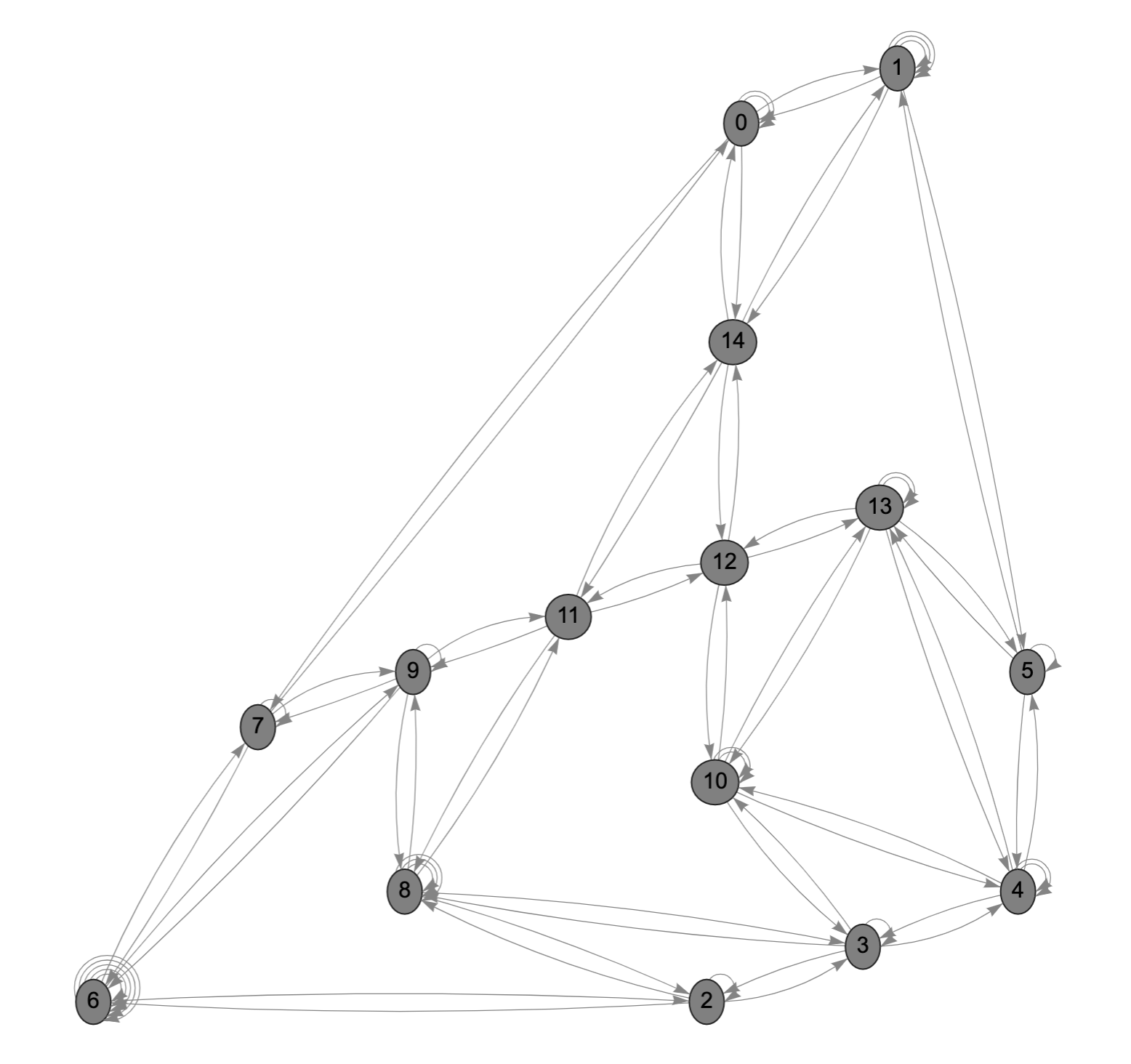}
    \caption{Curve quiver for $\frac{1}{11}(1,2,8)$-Hilb}
    \label{fig:1/128_0}
    \includegraphics[width=6.5in]{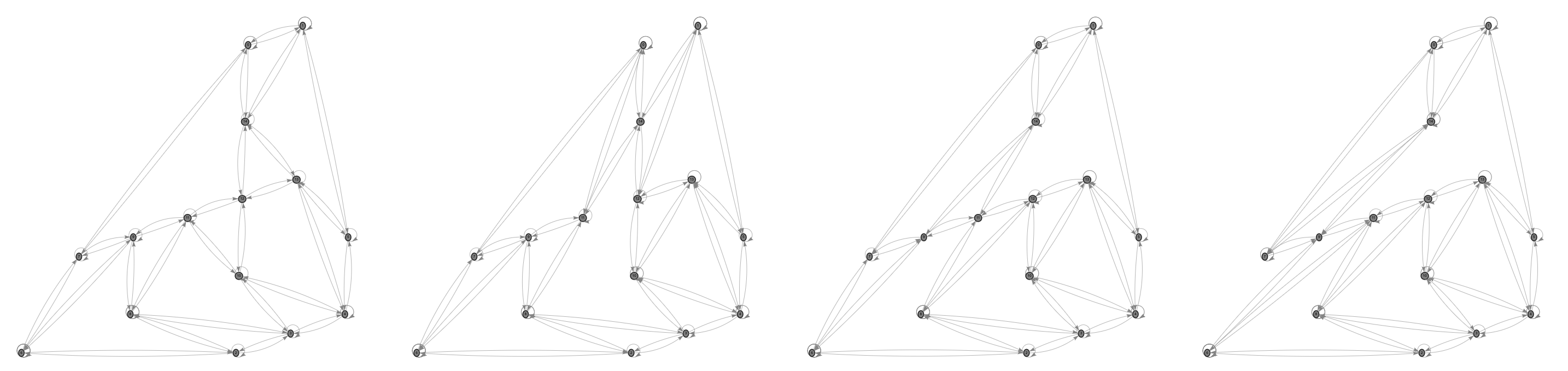}
    \caption{Other curve quivers for $\frac{1}{11}(1,2,8)$}
    \label{fig:1/128_1_to_4}
\end{figure}

For this data we find the following potential:
{\small
\begin{align*} &-x_{13,12}x_{12,11}x_{11,12}x_{12,13} + x_{12,10}x_{10,12}x_{12,14}x_{14,12} + x_{13,10}x_{10,12}x_{12,13} + x_{13,12}x_{12,10}x_{10,13} \\ 
&+ x_{14,0}x_{0,14}x_{14,12}x_{12,14} + x_{13,12}x_{12,13}x_{13,13} + x_{12,10}x_{10,10}x_{10,12} -x_{14,1}x_{1,14}x_{14,11}x_{11,14} + x_{14,0}x_{0,1}x_{1,14} \\ &+ x_{14,1}x_{1,0}x_{0,14} + x_{11,8}x_{8,11}x_{11,14}x_{14,11} + x_{11,9}x_{9,11}x_{11,12}x_{12,11} + x_{9,6}x_{6,9}x_{9,11}x_{11,14}x_{14,11}x_{11,9} \\ &+x_{9,7}x_{7,9}x_{9,9} + x_{9,6}x_{6,7}x_{7,9} + x_{9,7}x_{7,6}x_{6,9} + x_{11,8}x_{8,9}x_{9,11} + x_{11,9}x_{9,9}x_{9,11} + x_{11,9}x_{9,8}x_{8,11} \\ &+ x_{14,11}x_{11,12}x_{12,14} + x_{12,11}x_{11,14}x_{14,12}
\end{align*}
}
\end{example}

\begin{example}[A non-reducible potential]

In some cases our generation procedure results in a non-reducible potential. For $G=\frac{1}{12}(1, 3, 8)$ we perform several iterations of the potential-finding procedure.

\begin{figure}[h]
    \centering
    \includegraphics[width=4.5in]{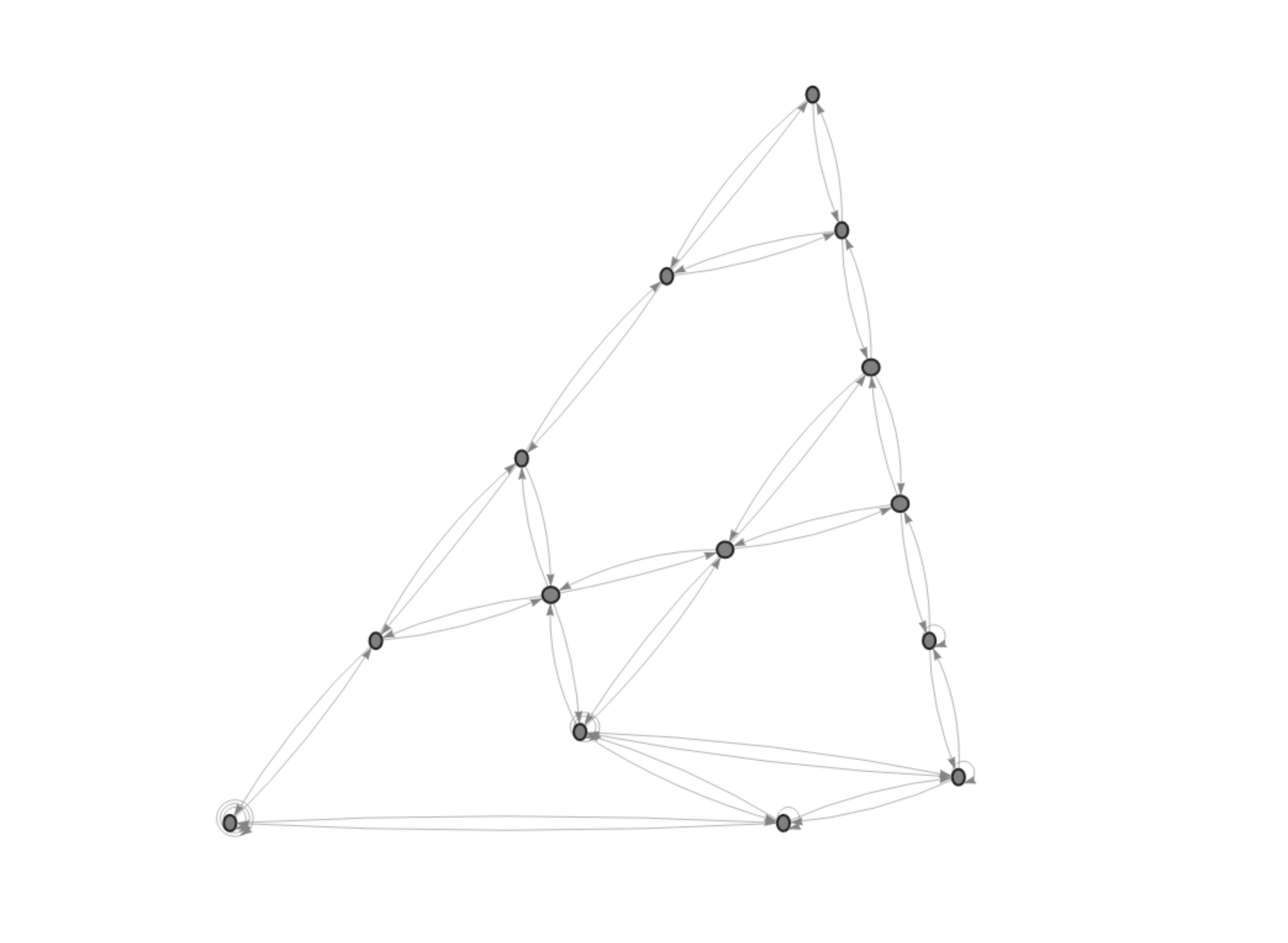}
    \caption{Curve quiver for $\frac{1}{12}(1,3,8)$-Hilb}
    \label{fig:1/12}
\end{figure}

The potential we produce after a few iterations has the following terms involving node 18 in the numbering from the \texttt{qwp\_mutations} application:
$$x_{11,18}x_{18,18}x_{18,11}\qquad x_{18,11}x_{11,11}x_{11,18}x_{18,15}x_{15,18}$$

Node 16 is adjacent to nodes 15 and 18. The mutation sequence then calls for a mutation at 16, which causes the two terms to be transformed into:
$$x_{18,11}x_{11,18}x_{18,16}x_{16,18}\qquad x_{18,11}x_{11,11}x_{11,18}x_{18,16}x_{16,15}x_{15,16}x_{16,18}$$

The next mutation in the sequence at node 18, and performing this mutation transforms the first term into a quadratic term. However, because this term is a subset of the second term it cannot be eliminated and so the reduction of the potential fails.
\end{example}

\subsection{Final musings} We conclude with some final thoughts and questions.

\subsubsection{Expectation from noncommutative geometry} Our understanding of the homological minimal model program \cite{wem_asp_14,wem_flo_18,dw_non_16} is that there is an expectation that the curve quiver of a crepant resolution, possibly equipped with a potential, should correspond to some natural object in noncommutative geometry (perhaps a kind of noncommutative deformation space for the given resolution). Mutations in this context should, as usual, describe the impact that flopping a curve has on this noncommutative object. In this sense it is reasonable to hunt for a suitable potential for curve quivers associated to the easiest resolution to construct -- the $G$-Hilbert scheme -- in this setting. 

\begin{question}
    Is there a natural interpretation of a (or the) quiver with (minimal) potential whose mutations classify flops of $G\hilb\C^3$ in terms of noncommutative geometry?
\end{question}

There are many other questions of interest one could pose: for instance, one could focus in on other special resolutions with notable moduli interpretations, such as iterated Hilbert schemes (or `Hilb of Hilb') \cite{iin_gnh_13,wor_wal_22}.

\begin{question}
    If $Y\to\C^3/G$ is a resolution with some additional equivariant structure (e.g.~an iterated Hilbert scheme), how does this manifest in a (or the) quiver with (minimal) potential whose mutations classify flops of $Y$?
\end{question}

One can regard the present work as a computational approach to some of these problems that, empowered by Prop.~\ref{prop:minimal}, bypasses certain aspects of the noncommutative geometry.

\subsubsection{Combinatorial questions}

We discussed the issues involved with verification, mostly stemming from the lack of an efficient graph isomorphism test. Our questions center on this issue in the specific context of curve quivers (from triangulations).

\begin{question} \label{conj:qs_eg} Suppose $Q$ is the curve quiver of a triangulation corresponding to a toric resolution $Y\to X$. Let $S$ be the set of curve quivers of resolutions related to $Y$ by flops, and let $\Gamma$ be the corresponding graph of resolutions linked by flops.
\begin{enumerate}
    \item Do all (or most) potentials on $Q$ with quiver set $S$ have exchange graph $\Gamma$?
    \item Do all (or most) quivers obtained by $Q$ from mutation with expected inward and outward vertex valencies agree? In other words, is valency testing sufficient to test for quiver isomorphism in this setting?
\end{enumerate}
\end{question}

We have yet to find a counterexample to either of these questions in our computations detailed above.

\subsubsection{Dimer models} As pointed out to us by Tom Ducat, little of the preceding needs to be tied to quotient singularities but could likely be extended to all Gorenstein toric singularities via the theory of dimer models \cite{iu_dim_16,cht_com_20}. It may even be cleaner to access Conjecture \ref{conj:main} in this wider context.

\begin{question}
    Is the extension of Conjecture \ref{conj:main} to all Gorenstein toric singularities true?
\end{question}

\bibliographystyle{acm}
\bibliography{bw}

\end{document}